\documentclass[10pt]{amsart}

\usepackage[colorlinks]{hyperref}
\usepackage{color,graphicx,shortvrb}
\usepackage[latin 1]{inputenc}

\usepackage[active]{srcltx} %SRC Specials for DVI search

\usepackage{enumerate}
\usepackage{amssymb}

\newtheorem{theorem}{Theorem}[section]

\newtheorem{corollary}[theorem]{Corollary}

\newtheorem{definition}[theorem]{Definition}
\newtheorem{lemma}[theorem]{Lemma}

\newtheorem{proposition}[theorem]{Proposition}
\newtheorem{remark}[theorem]{Remark}

\newcommand\K{\mathbb K}
\newcommand\N{\mathbb N}
\newcommand\R{\mathbb R}

\DeclareGraphicsExtensions{.jpg,.pdf,.png,.eps}

\begin{document}

\title[Weakly open convex combinations of slices and  C$^*$-Algebras]{Relatively weakly open convex combinations of slices and scattered C$^*$-Algebras }

\author{Julio Becerra Guerrero y Francisco J. Fern\'{a}ndez-Polo}

\address[J. Becerra Guerrero, F.J. Fern\'{a}ndez-Polo]{Departamento de An{\'a}lisis Matem{\'a}tico, Facultad de Ciencias, Universidad de Granada, 18071 Granada, Spain.}

\email{juliobg@ugr.es,  pacopolo@ugr.es, }

%\thanks{}



\date{}

\begin{abstract}

We prove that given a locally compact Hausdorff space, $K$, and a compact C$^*$-algebra, $\mathcal{A}$, the C$^*$-algebra $C(K, \mathcal{A})$ satisfies that every convex combination of slices of the closed unit ball is relatively weakly open subset of the closed unit ball if and only if $K$ is scattered and $\mathcal{A}$ is the $c_0$-sum of finite-dimensional C$^*$-algebras. We introduce and study Banach spaces which have property $(\overline{\hbox{\rm{P1}}})$, i. e.\smallskip

\emph{For every convex combination of slices $C$ of the unit ball of a Banach space $X$  and $x\in C$  there exists $W$ relatively weakly open set containing $x$, such that $W\subseteq \overline{C}$.}\smallskip

In the setting of general C$^*$-algebras we obtain a characterization of this property. Indeed, a C$^*$-algebra has property $(\overline{\hbox{\rm{P1}}})$ if and only if is scattered with finite dimensional irreducible representations. Some stability results for Banach spaces satisfying property $(\overline{\hbox{\rm{P1}}})$ are also given. As a consequence of these results we prove that  a real $L_1$-predual Banach space contains no isomorphic copy of $\ell_1$ if and only if it has property $(\overline{\hbox{\rm{P1}}})$.
\end{abstract}

\maketitle
\thispagestyle{empty}

\section{Introduction}\label{sec:intro}

In \cite{GGMS}, Ghoussoub, Godefroy, Maurey, and Schachermayer exhibited, perhaps for the first time, a remarkable geometrical property of the face of the positive elements in the unit ball of the classical Banach space $L_1[0,1]$. Indeed, it is proved in \cite[Remark IV.5, p. 48]{GGMS} that, for $\mathcal{F}:=\{ f\in L_1[0,1]:\ f\geq 0 \  \Vert f\Vert =1 \}$, any convex combination of a finite number of relatively weakly open subsets (in particular, slices)  of $\mathcal{F}$ is  a  relatively  weakly  open subset of $\mathcal{F}$.  It is well known that all relatively weakly open subsets of the unit ball of a infinite dimensional Banach space contains a convex combination of slices and that the reciprocal is not true. We will say that a Banach space $X$ has property $\hbox{\rm{(P1)}}$ if \smallskip

\emph{For every convex combination of slices $C$ of $B_X$ and $x\in C$  there exists $W$ relatively weakly open subset of the unit ball of $X$ containing $x$, such that $W\subseteq C$.}

\medskip

Quite recently, T.A. Abrahamsen and V. Lima obtained that given a scattered locally compact space $K$, the space of continuous functions on $K$ vanishing at infinity,  $C_0(K, \mathbb{C})$, satisfy property (P1) \cite[Theorem 2.3]{AbrLim}. On the other hand, R. Haller, P. Kuuseok and M. P\~{o}ldvere showed that $K$ is scattered whenever $C_0(K,\mathbb{R})$ has property (P1) \cite[Theorem 3.1]{HalKuuPol}. In \cite{AbrBecHalLimPol}, the authors introduce another geometric property, named $\hbox{\rm{(co)}}$ (see Definition \ref{d (co)}), and show that if a finite dimensional normed space $X$ has this property, then for any scattered locally compact Hausdorff space, $K$, the space $C_0(K,X)$ has property $\hbox{\rm{(P1)}}$.\smallskip

In view of the results obtained in the aforementioned works, the study of property $\hbox{\rm{(P1)}}$ seems to be reduced to a very small family of Banach spaces that except in the case of finite dimension they have in common that their dual has a structure similar to $\ell_1$. It is natural to wonder if those Banach spaces whose duality has a certain similarity to $\ell_1$ have property $\hbox{\rm{(P1)}}$. In this sense, we introduce a property weaker than property $\hbox{\rm{(P1)}}$ that we will be called  $(\overline{\hbox{\rm{P1}}})$ and reads as follows:

\smallskip

\emph{For every convex combination of slices $C$ of $B_X$  and $x\in C$  there exists $W$ relatively weakly open subset of the unit ball of $X$ containing $x$, such that $W\subseteq \overline{C}$.}

\medskip

This paper is organized as follows. In Section 2 we present some stability results for properties  $\hbox{\rm{(P1)}}$ and $(\overline{\hbox{\rm{P1}}})$. For instance, we prove that these properties  are preserved by contractive projections (Proposition \ref{1-complementado}) and by $c_0$-sums (Proposition \ref{c0sum}).  We also show that every Banach space $X$ such that $X^*=\bigoplus_{\gamma \in \Gamma}^{\ell_1} X_\gamma $ for a family of finite-dimensional Banach spaces $\{X_\gamma \}_{\gamma \in \Gamma}$ with the $\hbox{\rm{(co)}}$ property, has property $(\overline{\hbox{\rm{P1}}})$ (Theorem \ref{thm-sum}). As a consequence, every isometric preduals of $\ell_1$ has property  $(\overline{\hbox{\rm{P1}}})$. This provides examples of Banach spaces with property $(\overline{\hbox{\rm{P1}}})$ that are not isomorphic to a complemented subspace of any $C(K)$-space (see \cite[Corollary 1]{BenLin}). We conclude this section showing that for real Banach spaces, property  $(\overline{\hbox{\rm{P1}}})$ is hereditary with respect to subspaces that are ideals (Proposition \ref{locally-complemented-ideal}). As a consequence of the aforementioned results, we obtain that for a real $L_1$-predual Banach space property $(\overline{\hbox{\rm{P1}}})$  is  equivalent  to containing no isomorphic copy of $\ell_1$ (Theorem \ref{L_1-predual}).\smallskip

Section \ref{(co)} is mainly devoted to establish that finite dimensional C$^*$-algebras have property $\hbox{\rm{(co)}}$ (see Definition \ref{d (co)} and Theorem \ref{t Mn is (co)}).  One of the key ingredients in the proof of Theorem \ref{t Mn is (co)} is the perturbation theory in C$^*$-algebras  \cite{Bha}. Concretely, in Theorem \ref{t continuity of spectral resolutions} we generalize some classical results of C. Davis in \cite{Dav} for non-necessarily  self-adjoint elements.\medskip

In Section \ref{sect: scattered} we will focus on the class of scattered C$^*$-algebras introduced and characterized by H.E. Jensen in \cite{Jen77, Jen78}. Since abelian C$^*$-algebras are scattered if and only if satisfy property $\hbox{\rm{(P1)}}$ \cite{AbrLim, HalKuuPol}, it is natural to study properties $\hbox{\rm{(P1)}}$ and $(\overline{\hbox{\rm{P1}}})$ in the setting of general C$^*$-algebras. We show that C$^*$-algebras satisfying property $\hbox{\rm{(P1)}}$ (or $(\overline{\hbox{\rm{P1}}})$) turn out to be scattered (Proposition \ref{p_(P1) implies scattered}). We also prove  that given a locally compact Hausdorff space $K$, and a compact C$^*$-algebra $\mathcal{A}$, the C$^*$-algebra $C(K, \mathcal{A})$ has property $\hbox{\rm{(P1)}}$ if and only if $K$ is scattered and $\mathcal{A}$ is the $c_0$-sum of finite-dimensional C$^*$-algebras (Theorem \ref{c C(K,Mn) has P1}). Finally, we provide a characterization of C$^*$-algebras satisfying property $(\overline{\hbox{\rm{P1}}})$ as those being scattered and having only finite dimensional representations (see Theorem \ref{t caracterizacion P1 barra}).\medskip

\section{Stability results for properties $\hbox{\rm{(P1)}}$ and $(\overline{\hbox{\rm{P1}}})$} \label{sec:ckx-spaces}

\medskip

We shall first introduce some notation. Let us consider $X$ a real or complex Banach space. We will denote by $B_X$ (respectively, $S_X$)  the closed unit ball (respectively, unit sphere) of the Banach space $X$. Given $C$ a norm-closed convex subset of $X$, for every $f\in B_{X^*}$ and $\varepsilon >0$ we define a \emph{slice of} $C$ by $$S(C,f,\varepsilon):=\{x\in C:\ \mathfrak{Re} f( x)>\sup_{c\in C}|f(c)|-\varepsilon\}.$$

\medskip

Secondly we will present the definitions object of study in this paper.\smallskip

\begin{definition}\label{defn:cwos}

Let us consider the following properties for a Banach space $X$:
\begin{enumerate}
\item[$\hbox{\rm{(P1)}}$] For every convex combination of slices $C$ of $B_X$ and $x\in C$  there exists $W$ relatively weakly open subset of $B_X$ containing $x$, such that $W\subseteq C$.

\item[$(\overline{\hbox{\rm{P1}}})$] For every convex combination of slices $C$ of $B_X$  and $x\in C$  there exists $W$ relatively weakly open subset of $B_X$ containing $x$, such that $W\subseteq \overline{C}$.
\end{enumerate}

\end{definition}

\medskip

\begin{definition}\label{d (co)}

A Banach space, $X$, is said to have the  property $\hbox{\rm{(co)}}$ if for every $n\in \mathbb{N}$, given $x_1,\ldots,x_n \in B_X$, $\lambda_1,\ldots,\lambda_n>0$ with $\sum_{i=1}^{n} \lambda_i=1$ and $\varepsilon >0$ there exist $\delta>0$ and continuous functions $\Phi_i:B(x_0,\delta)\cap B_X\to B(x_i,\varepsilon)\cap B_X$, where $x_0=\sum_{i=1}^{n} \lambda_i x_i$, satisfying $y=\sum_{i=1}^{n} \lambda_i\Phi_i(y)$ for  every $y\in B(x_0,\delta)$.

\end{definition}

Throughout this section we will reveal  new results in the setting of general Banach spaces concerning these three properties. We begin by showing that all these properties are preserved by contractive projections.\smallskip

\begin{proposition}\label{1-complementado}
Norm-one complemented subspaces of a Banach space inherit each of properties $\hbox{\rm{(P1)}}$,  $(\overline{\hbox{\rm{P1}}})$ and $\hbox{\rm{(co)}}$.
\end{proposition}
\begin{proof}

Let $X$ be a Banach space, let $P:X\to X$ be a contractive projection and $Y=P(X)$.
Let $\{S(B_Y,f_i,\varepsilon_i)\}_{i=1}^k$ be slices of $B_{Y}$,let $\lambda_i > 0$ with $\sum_{i=1}^k
\lambda_i = 1$, and consider the convex combination of these slices
\begin{equation*}
    C = \sum_{i=1}^k \lambda_i S(B_Y,f_i,\varepsilon_i).
\end{equation*}
Take $y = \sum_{i=1}^k \lambda_i y_i $ in $C$ with $y_i \in S(B_Y,f_i,\varepsilon_i)$. We consider the slices of $B_X$, $\{S(B_X,P^*(f_i),\varepsilon_i)\}_{i=1}^k$, and $ \widetilde{C}:= \sum_{i=1}^k \lambda_i S(B_X,P^*(f_i),\varepsilon_i).$ We can assume that $X$ has property $\hbox{\rm{(P1)}}$. Since $y\in \widetilde{C}$,  there exists $W$ relatively weakly open subset of $B_X$ containing $y$ such that $W\subseteq \widetilde{C}$. Now, $P(W)$ is a relatively weakly open subset of $B_Y$ satisfying  $y\in P(\widetilde{C})\subseteq C$.

In the case that property $(\overline{\hbox{\rm{P1}}})$ is satisfied, the proof is similar and the case of the $\hbox{\rm{(co)}}$ property is trivial.
\end{proof}

\medskip

The following is a frequently used technical result.

\begin{lemma}\label{prime}
Let $X$ be a Banach space. Consider $f\in B_{X^*}, \varepsilon\in\mathbb R^+$ and $x\in S(B_X,f,\varepsilon)$. Then, there exists $\widetilde{\varepsilon}, \ \rho >0$  such that for all $g\in B_{X^*}$ with $\Vert f-g\Vert <\rho$ we have that $$x\in S(B_X,g,\widetilde{\varepsilon})\subseteq S(B_X,f,\varepsilon).$$
\end{lemma}

\medskip

Let $\Gamma$ be a set, and let  $\{X_\gamma : \gamma \in \Gamma\}$ be a family of Banach spaces indexed by $\Gamma$.\smallskip

We recall that the $c_0$-sum of the family $\{X_\gamma : \gamma \in \Gamma\}$, denoted by  $\bigoplus_{\gamma \in \Gamma}^{c_0} X_\gamma $, is the Banach
space $$\bigoplus_{\gamma \in \Gamma}^{c_0} X_\gamma :=\{ (x_\gamma ):x_\gamma \in X_\gamma ,\
\lim_{\gamma \in \Gamma }\{\Vert x_\gamma \Vert_\gamma \}=0 \} $$ and $\Vert (x_\gamma )\Vert =
\sup \{\Vert x_\gamma \Vert_\gamma :\gamma \in  \Gamma \}$ for each $(x_\gamma )\in \bigoplus_{\gamma \in \Gamma}^{c_0} X_\gamma $.\smallskip

The $\ell_1$-sum of the family $\{X_\gamma : \gamma \in \Gamma\}$, denoted by  $\bigoplus_{\gamma \in \Gamma}^{\ell_1} X_\gamma $, is the Banach space $$\bigoplus_{\gamma \in \Gamma}^{\ell_1} X_\gamma:=\{ (x_\gamma ):x_\gamma \in X_\gamma , \gamma \in \Gamma ,\ \sum_{\gamma \in \Gamma}\Vert x_\gamma \Vert_\gamma <\infty \} $$ and $\Vert (x_\gamma )\Vert = \sum_{\gamma \in \Gamma}\Vert x_\gamma \Vert_\gamma$. \smallskip

Finally, the $\ell_{\infty}$-sum of the family $\{X_\gamma : \gamma \in \Gamma\}$, $\bigoplus_{\gamma \in \Gamma}^{\ell_{\infty}} X_\gamma $, is the Banach space $$\bigoplus_{\gamma \in \Gamma}^{\ell_{\infty}} X_\gamma :=\{ (z_\gamma ):z_\gamma \in X_\gamma  ,\ \sup \{\Vert x_\gamma \Vert_\gamma \}<\infty \} $$ and $\Vert (z_\gamma )\Vert = \sup \{\Vert z_\gamma\Vert_\gamma :\gamma \in \Gamma \}$ for each $(z_\gamma )\in \bigoplus_{\gamma \in \Gamma}^{\ell_{\infty}} X_\gamma $.\smallskip

It is also well known that $(\bigoplus_{\gamma \in \Gamma}^{c_0} X_\gamma)^*= \bigoplus_{\gamma \in \Gamma}^{\ell_1} X_\gamma^* $ and  $(\bigoplus_{\gamma \in \Gamma}^{\ell_1} X_\gamma)^* =\bigoplus_{\gamma \in \Gamma}^{\ell_{\infty}} X_\gamma^*.$  For each $\mathcal R\subseteq \Gamma $, we denote by $P_{\mathcal R}^0$ (respectively, $P_{\mathcal R}^1$, $P_{\mathcal R}^{\infty}$) the canonical projection of $\bigoplus_{\gamma \in \Gamma}^{c_0} X_\gamma $ (respectively, $\bigoplus_{\gamma \in \Gamma}^{\ell_1} X_\gamma^* $, $\bigoplus_{\gamma \in \Gamma}^{\ell_{\infty}} X_\gamma^{**} $)  onto $\bigoplus_{\gamma \in \mathcal{R}}^{c_0} X_\gamma $ (respectively, $\bigoplus_{\gamma \in \mathcal{R}}^{\ell_1} X_\gamma^* $, $\bigoplus_{\gamma \in \mathcal{R}}^{\ell_{\infty}} X_\gamma^{**} $).

\medskip

The following result generalizes Theorem 5.2 in \cite{AbrBecHalLimPol} since every finite dimensional Banach space with the $\hbox{\rm{(co)}}$ property  has $\hbox{\rm{(P1)}}$ \cite[Proposition 2.2]{AbrBecHalLimPol}.\medskip

\begin{proposition}\label{c0sum} Let $\Gamma$ be a set, $\{X_\gamma : \gamma \in \Gamma\}$ a family of Banach spaces indexed by $\Gamma$. Then $Z:=\bigoplus _{\gamma \in \Gamma}^{c_0} X_\gamma $ has property $\hbox{\rm{(P1)}}$ (respectively, $(\overline{\hbox{\rm{P1}}})$) if and only if $X_\gamma$ has property $\hbox{\rm{(P1)}}$  (respectively, $(\overline{\hbox{\rm{P1}}})$) for every $\gamma \in \Gamma$.
\end{proposition}

\begin{proof} The only if part is given by Proposition \ref{1-complementado}. Now, let $\{S(B_Z,f_i,\varepsilon_i)\}_{i=1}^n$   be slices of $B_{Z}$, $f_1,..,f_n\in S_{Z^*}$, let $\lambda_i > 0$ with $\sum_{i=1}^n \lambda_i = 1$, and consider the convex combination of these slices \begin{equation*}
C = \sum_{i=1}^n \lambda_i S(B_Z, f_i,\varepsilon_i).
\end{equation*}

Let $z = \sum_{i=1}^n \lambda_i z_i \in C$ with $z_i \in S(B_Z,f_i,\varepsilon_i)$. Our goal is to find a non-empty relatively weakly open subset of $B_Z$ containing $z$ that is contained in $C$. By Lemma \ref{prime}, given $i\in \{1,.., n \}$ and $\varepsilon _i$, there exists $\widetilde{\varepsilon}_i, \ \rho _i >0$  such that for all $g\in B_{Z^*}$ with $\Vert f_i-g\Vert <\rho _i$ we have that $$z_i\in S(B_Z,g,\widetilde{\varepsilon}_i)\subseteq S(B_Z,f_i,\varepsilon _i).$$

Fix $\mathcal R \subseteq \Gamma$ a finite subset with $N\in \mathbb{N}$ elements, such that $\Vert f_i-P_{\mathcal R}^1(f_i)\Vert <\rho _i$ for all $i\in \{1,.., n \}$, where $P_{\mathcal R}^1$ is the natural projection onto $\bigoplus _{\gamma \in \mathcal{R}}^{\ell_1} X^*_\gamma$. Then we have that $$z \in \sum_{i=1}^n \lambda_i S(B_Z,P_{\mathcal R}^1(f_i),\widetilde{\varepsilon}_i)\subseteq C .$$
Since for $i\in \{1,.., n \}$, $\mathfrak{Re}P_{\mathcal R}^1(f_i)(z_i)> \| P_{\mathcal R}^1(f_i)\|- \widetilde{\varepsilon}_i$, we set $\varepsilon >0$ such that $$\mathfrak{Re}P_{\mathcal R}^1(f_i)(z_i)-N \varepsilon> \|P_{\mathcal R}^1(f_i)\|- \widetilde{\varepsilon}_i $$  for all $i\in \{1,\ldots,n\}$.\smallskip

Given $i\in \{1,.., n \}$ and $j\in \mathcal R$, we consider the following slices in the closed unit ball of the Banach space $X_j$ $$S_i^j:=S(B_{X_j},P_{\mathcal R}^1(f_i)(j),\|P_{\mathcal R}^1(f_i)(j)\|-\mathfrak{Re}P_{\mathcal R}^1(f_i)(z_i(j))+\varepsilon).$$ Now, for every $j\in \mathcal R$, we have that $$z(j)= \sum_{i=1}^n \lambda_i z_i(j)\in \sum_{i=1}^n \lambda_i S_i^j,$$ and since $X_j$ has property $\hbox{\rm{(P1)}}$, there exists $W_j$ relatively weakly open subset of $B_{X_j}$ containing $z(j)$, such that $ W_j \subseteq \sum_{i=1}^n \lambda_i S_i^j$. We define the set $$W:=(\prod_{\j\in \mathcal{R}} W_j)\times B_{(I-P_{\mathcal R}^0)(Z)},$$ which is clearly a relatively weakly open subset $B_Z$ containing $z$.\smallskip

We will finish the proof by showing that $W\subseteq C$. Indeed, given $y\in W$ we have that $y(j)\in W_j$ for every $j\in \mathcal{R}$ and thus $y(j)=\sum_{i=1}^{n} \lambda_i y(j)_i$ where each $y(j)_i$ belong  to $S_i^j$. Therefore we can define, for each $i\in \{1,\ldots,n\}$, the element $y_i\in Z$ given by $$y_i(\gamma):=\left\{%
\begin{array}{ll}
    y(\gamma)_i, & \gamma \in \mathcal{R} \\
   y(\gamma), &  \gamma \in \Gamma\backslash\mathcal{R}\\

\end{array}%
\right.$$  satisfying $y=\sum_{i=1}^{n} \lambda_i y_i$ and $$ \mathfrak{Re}P_{\mathcal R}^1(f_i) y_i= \sum_{j\in \mathcal{R}}\mathfrak{Re}P_{\mathcal R}^1(f_i)y(j)_i>\sum_{j\in \mathcal{R}} (\mathfrak{Re}P_{\mathcal R}^1(f_i)(z_i(j))-\varepsilon)=$$ $$\mathfrak{Re}P_{\mathcal R}^1(f_i)(z_i)- N\varepsilon>\| P_{\mathcal R}^1(f_i) \|-\tilde{\varepsilon}_i,$$ so thus $\displaystyle y\in  \sum_{i=1}^n \lambda_i S(B_Z,P_{\mathcal R}^1(f_i),\widetilde{\varepsilon}_i)\subseteq C$.

In case we are dealing  property $(\overline{\hbox{\rm{P1}}})$ the proof is analogous.\end{proof}

\medskip

The following result provides new examples of Banach spaces of type $C (K, X)$ with property $\hbox{\rm{(P1)}}$, where $X$ is an infinite dimensional Banach space.

\medskip

\begin{corollary}\label{slices-for-C_0}
Let $K$ be scattered compact Hausdorff space and let $\{X_\gamma \}_{\gamma \in \Gamma}$ be a family of finite-dimensional Banach spaces with $\hbox{\rm (co)}$ property. Then  $C(K,(\bigoplus_{\gamma \in \Gamma}^{c_0} X_\gamma ))$ has property $\hbox{\rm{(P1)}}$.
\end{corollary}
\begin{proof}
For every family of Banach spaces  $\{X_\gamma \}_{\gamma \in \Gamma}$ , the equality $C(K,(\bigoplus_{\gamma \in \Gamma}^{c_0} X_\gamma ))=\bigoplus_{\gamma \in \Gamma}^{c_0} C(K,X_\gamma )$  holds. Since $X_\gamma$ has the $\hbox{\rm (co)}$ property, by \cite[Theorem 2.5]{AbrBecHalLimPol}, $C(K,X_\gamma )$ has property $\hbox{\rm{(P1)}}$ for every $\gamma \in \Gamma$. Finally, Proposition \ref{c0sum} applies.\end{proof}

\medskip

\begin{theorem}\label{thm-sum}
Let $\{X_\gamma \}_{\gamma \in \Gamma}$ be a family of finite-dimensional Banach spaces with $\hbox{\rm (co)}$
property. Let $Z$ be a Banach space such that $Z^*=\bigoplus_{\gamma \in \Gamma}^{\ell_1} X_\gamma $. Then $Z$ has property $(\overline{\hbox{\rm{P1}}})$.
\end{theorem}
\begin{proof}

As before, let $C$ be a convex combination of slices of the unit ball of $Z$,
\begin{equation*}
    C = \sum_{i=1}^n \lambda_i S(B_Z, f_i,\varepsilon_i).
\end{equation*}

Fix  $z_0 = \sum_{i=1}^k \lambda_i z_i \in C$ with $z_i \in S(B_Z,f_i,\varepsilon_i)$. Our goal is to find a
non-empty relatively weakly open subset of $B_Z$ containing $z_0$, that is contained in $\overline{C}$.\smallskip

Let $ \eta $ be a positive number satisfying $\mathfrak{Re}f_i(z_i)-3\eta>\|f_i\|-\varepsilon_i$ for all $i\in \{1,\ldots,n\}$. Associated to $\eta$ we can find a finite set $\mathcal R \subseteq \Gamma$ such that $\|f_i-P_{\mathcal R}^1(f_i)\|<\eta$ for all $i\in \{1,\ldots,n\}$ and let $N\in \mathbb{N}$ denote the cardinal of $\mathcal{R}$.\smallskip

By hypothesis $Z^*=\bigoplus_{\gamma \in \Gamma}^{\ell_1} X_\gamma $ and hence $Z^{**}=\bigoplus_{\gamma \in \Gamma}^{\ell_{\infty}} X_\gamma $. We will represent the elements of $Z$ as elements in $Z^{**}$.\smallskip

Since $X_\gamma$ have the $\hbox{\rm{(co)}}$ property for every $\gamma \in \Gamma$, given $\frac{\eta}{3N}>0$ (and $z_0(\gamma)\in X_{\gamma}$) there exist $\delta_{(i,\gamma)}>0$ and continuous functions $\Phi_{(i,\gamma)}:B(z(\gamma),\delta_{(i,\gamma)})\cap B_{X_\gamma} \to B(z_i(\gamma),\frac{\eta}{3N})\cap B_{X_\gamma}$ such that for every $y\in B(z_0(\gamma ),\delta_{(i,\gamma )})\cap B_{X_\gamma }$ we have $y= \sum_{i=1}^k \lambda_i \phi _{(i,\gamma )}(y)$, where $i\in \{1,\ldots,n\}$.\medskip

We take $\delta :=\min \{\{\delta_{(i,\gamma )}:\ 1\leq i\leq k,\ \gamma \in \mathcal R  \}, \frac{\eta }{3N} \}$.\smallskip

For every $\gamma \in \mathcal R$, we can choose a $\delta/9$-net of the unit sphere of $X_{\gamma}$, $(\varphi_j(\gamma ))_{j=1}^{M_\gamma }$, where $M_\gamma$ is the (finite) cardinal of the net. We can extend this functionals to $S_{Z^*}$ in a natural way, defining $\varphi_{(j,\gamma )}(\xi )= \delta _{\gamma ,\xi}\varphi_j(\gamma )$ for each $\xi\in \mathcal R$ (and zero elsewhere), where in this case $\delta _{\gamma ,\xi }$ is the Kronecker delta.\smallskip

We now define the relatively weakly open subset of $B_Z$

\begin{equation*}
    \mathcal{U} =
    \left\{
      z \in B_Z :
        |g(z-z_0)| < \delta/9,
        g\in H
    \right\},
\end{equation*}

where $H:=\{\varphi_{(j,\gamma )}:\gamma \in \mathcal R , 1 \leq j \leq M_\gamma \}\cup \{P_{\mathcal R}^1(f_i):1\leq i\leq k\}.$\smallskip

We consider $\widetilde{C}:=\sum_{i=1}^n \lambda_i S(B_{Z^{**}},f_i,\varepsilon_i)$ and $$\widetilde{\mathcal{U}} =     \left\{
      z \in B_{Z^{**}} :
        |g(z-z_0)| < \delta/9,
        g\in H
    \right\}.$$

We claim that $\widetilde{\mathcal{U}}\subseteq \widetilde{C}$.\medskip

Let $z$ be in $\widetilde{\mathcal{U}}$. We will define $z_i \in S(B_{Z^{**}},f_i,\varepsilon_i)$, $i=1,2,\ldots,k$, and show that $z$ can be written as $z = \sum_{i=1}^k \lambda_i z_i \in \widetilde{C}$.\smallskip

Since $z\in \widetilde{\mathcal{U}}$,  for $\gamma \in \mathcal R$ and $1 \leq j \leq M_\gamma $, $$\vert \varphi_{(j,\gamma )}(z_0-z)\vert < \delta /9 ,$$ and hence $\Vert z_0(\gamma )-z(\gamma )\Vert_{X_\gamma }<\frac{\delta}{3}$.\smallskip

Having in mind that   $z_0(\gamma ) = \sum_{i=1}^n \lambda_i z_i(\gamma )$ for $\gamma \in \mathcal R$,  $X_\gamma $ has the $\hbox{\rm{(co)}}$ property, and $\Vert z_0(\gamma )-z(\gamma )\Vert _{X_\gamma}<\frac{\delta}{3}$, we have that  $$z(\gamma ) = \sum_{i=1}^n \lambda_i \Phi _{(i,\gamma )}(z_i(\gamma ))\ \  \mbox{with}\  \   \Vert
z_i(\gamma )-\Phi _{(i,\gamma )}(z_i(\gamma ))\Vert _{X_\gamma
}<\frac{\eta}{3N}.$$

We define $\tilde{z}_i\in B_{Z^{**}}$ by $\tilde{z}_i(\gamma)=\Phi _{(i,\gamma )}(z_i(\gamma ))$  whenever $\gamma \in \mathcal R$ and $\tilde{z}_i(\gamma )=z(\gamma )$ otherwise.\smallskip

It is clear that $z = \sum_{i=1}^n \lambda_i \tilde{z}_i$ and it only remains to show that $\tilde{z}_i\in
S(B_{Z^{**}},f_i,\varepsilon_i)$.\smallskip

We have
  \begin{equation*}
    |P_{\mathcal R}^1(f_i)(\tilde{z}_i - z_i)|
    \leq \sum_{\gamma \in \mathcal R }|P_{\mathcal R}^1(f_i)(\gamma )(\tilde{z}_i(\gamma ) - z_i(\gamma ))|
  \end{equation*}
  \begin{equation*}
    = \sum_{\gamma \in \mathcal R }|P_{\mathcal R}^1(f_i)(\gamma )(\Phi _{(i,\gamma )}(z_i(\gamma )) -
    z_i(\gamma ))|\leq N \Vert z_i(\gamma )-\Phi
_{(i,\gamma )}(z_i(\gamma ))\Vert _{X_\gamma } < N \frac{\eta}{3N}
<\frac{\eta}{3}
  \end{equation*}
Since $\|f_i-P_{\mathcal R}^1(f_i)\| < \eta$,  we have $|f_i(z_i) - P_{\mathcal R}^1(f_i)(z_i)| < \eta$ and $|f_i(\tilde{z}_i) - P_{\mathcal R}^1(f_i)(\tilde{z}_i)| < \eta$.  Hence $|f_i(\tilde{z}_i - z_i)| < 3\eta$  so that  \begin{equation*}
   \Re f_i(\tilde{z}_i) \ge \Re f_i(z_i) - 3\eta
    > 1 - \varepsilon_i,
  \end{equation*}
and we are done.

Finally, for each $y\in U$, we have that $y = \sum_{i=1}^n \lambda_i \tilde{z}_i$ with $\tilde{z}_i\in S(B_{Z^{**}},f_i,\varepsilon_i)$. Since $B_Z$ is $w^*$-dense in $B_{Z^{**}}$, we have that $y$ is the $w^*$-limit in $Z^{**}$ of the net $(\sum_{i=1}^n \lambda_i x_{(i,\alpha )})$ with $x_{(i,\alpha )}\in S(B_Z,f_i,\varepsilon_i)$. So it follows that $y$ is the $w$-limit of $(\sum_{i=1}^n \lambda_i x_{(i,\alpha )})$, and since $C$ is convex subset of $B_Z$, we concluded that $y\in \overline{C}$. \end{proof}

\medskip

\begin{remark}
Looking at the proof of Theorem \ref{thm-sum}, we realize that, if $Z$ is a Banach space such that $Z^*=\bigoplus_{\gamma \in \Gamma}^{\ell_1} X_\gamma $ for a family of finite-dimensional Banach spaces $\{X_\gamma \}_{\gamma \in \Gamma}$ with property $\hbox{\rm (co)}$, then every convex combination of $w^*$-slices of the unit ball of $Z^{**}$ is a relatively $w^*$-open subset of $B_{Z^{**}}$. In the particular case that  $X_\gamma =\R$ for every  $\gamma \in \Gamma$ this result is obtained in \cite[Proposition 4.5]{GiMaRu}.
\end{remark}

\medskip

The latter theorem  will be decisive in order to characterize C$^*$-algebras with property $(\overline{\hbox{\rm{P1}}})$ (see Theorem \ref{t caracterizacion P1 barra}). Having in mind that $\mathbb{C}$ and $\mathbb{R}$ have the $\hbox{\rm{(co)}}$ property (Proposition 2.3 in \cite{AbrBecHalLimPol}) we can establish the following result.\smallskip

\begin{corollary} Every isometric predual of $\ell_1$ has property $(\overline{\hbox{\rm{P1}}})$.

\end{corollary}

\medskip

The following reformulation of property $(\overline{\hbox{\rm{P1}}})$ will be useful in succeeding results.

\medskip

\begin{lemma}\label{P1-bidual}
Let $X$ be a Banach space. Then $X$ has property $(\overline{\hbox{\rm{P1}}})$ if and only if for every convex 	 combination of $w^*$-slices $C$ of $B_{X^{**}}$ and $x\in C \cap X$  there exists $W$ relatively $w^*$-open neighborhood of $x$, such that $W\subseteq \overline{C}^{w^*}$.
\end{lemma}
\begin{proof} Assume first that $X$ has property $(\overline{\hbox{\rm{P1}}})$.  We consider $C_{X^{**}}$ a convex combination of $w^*$-slices of $B_{X^{**}}$ and $x_0\in C_{X^{**}} \cap X$. Now  $C_X=C_{X^{**}}\cap B_X$  is a convex combination of slices of $B_X $ and $x_0\in C_X $. Since $X$ has property $\overline{(P1)}$, there is a relatively weakly open subset $U$ of $B_X$ with $x_0\in U\subseteq \overline{C}_X$. Then $U$ contains a
set $V$ of the form $$\{x\in B_X:\vert g_i(x-x_0)\vert <\delta\, \, \, \forall i=1,...,n\},$$ for suitable $\delta>0$, $n\in \N $, and $g_1,...,g_n\in X^*$. Set $$V^{**}:=\{z\in B_{X^{**}}:\vert g_i(z-x_0)\vert <1\, \, \, \forall i=1,...,n\}.$$ Since $V^{**}$ is relatively $w^*$-open in $B_{X^{**}}$, and $B_X$ is $w^*$-dense in $B_{X^{**}}$, the set $V$ ($=V^{**}\cap B_X$) is $w^*$-dense in $V^{**}$. Having in mind that $x_0\in V\subseteq \overline{C}_X$, it follows that $$x_0\in V^{**}= \overline{V}^{w^*}\subseteq \overline{C}_X^{w^*}\subseteq
\overline{C}_{X^{**}}^{w^*}.$$
	
We  assume now that $X^{**}$ has property that for every convex combination of $w^*$-slices $C_{X^{**}}$ of $B_{X^{**}}$ and $x\in C_{X^{**}}\cap X$  there exists $W$ a relatively $w^*$-open neighborhood of $x$, such that $W\subseteq \overline{C}_{X^{**}}^{w^*}$. We claim that $X$ has property $(\overline{\hbox{\rm{P1}}})$.\smallskip
	
Let $\{S(B_X,f_i,\varepsilon_i)\}_{i=1}^k$ be slices of $B_{X}$, let $\lambda_i > 0$ with $\sum_{i=1}^k \lambda_i = 1$, and consider the convex combination of these slices $C_X=\sum_{i=1}^k \lambda_i S(B_X,f_i,\varepsilon_i)$ and $x_0\in C_X$. Put $C_{X^{**}}=\sum_{i=1}^k \lambda_i S(B_{X^{**}},f_i,\varepsilon_i)$ convex combination of $w^*$-slices of $B_{X^{**}}$. Since $x_0\in C_{X^{**}}$, by the assumption, there exists $W$ a relatively $w^*$-open neighborhood of $x$, such that $x\in W\subseteq \overline{C}_{X^{**}}^{w^*}$. Since $S(B_{X^{**}},f_i,\varepsilon_i)$ is relatively $w^*$-open in $B_{X^{**}}$, and $B_X$ is $w^*$-dense in $B_{X^{**}}$, the set $S(B_X,f_i,\varepsilon_i)$  is $w^*$-dense in $S(B_{X^{**}},f_i,\varepsilon_i)$ for every $i=1,...,n$. Therefore,  $\overline{C}_{X^{**}}^{w^*}=\overline{C}_{X}^{w^*}$ and $x\in W\subseteq \overline{C}_{X}^{w^*}$. Now, $x\in V$ where $V:= W\cap B_X$ is a relatively weakly open subset of $B_X$ and for $y\in V$ we have that $y = \sum_{i=1}^k \lambda_i z_i$ with $z_i\in S(B_{X^{**}},f_i,\varepsilon_i)$. Since $B_X$ is $w^*$-dense in $B_{Z^{**}}$, we have that $y$ is the $w^*$-limit of $\{\sum_{i=1}^k \lambda_i x_{(i,\alpha )}\}$ with $x_{(i,\alpha )}\in S(B_X,f_i,\varepsilon_i)$ in $X^{**}$. So it follows that $y$ is the $w$-limit of $\{\sum_{i=1}^k \lambda_i x_{(i,\alpha )}\}$, and since $C_X$ is convex subset of $B_X$, we conclude that $y\in \overline{C}_X$.\end{proof}

\medskip

Let $X$ be  a real Banach space and $Y$ a subspace of $X$. We recall that $Y$ is an \emph{ideal } in $X$ if $Y^\perp$, the annihilator of $Y$ in $X^*$, is the kernel of a norm one projection on $X^*$. We also recall that a closed subspace $Y$ of a Banach space $X$ is said to be \emph{locally $1$-complemented} in $X$ if, for every finite dimensional subspace $E$ of $X$ and every $\varepsilon > 0$, there exists a linear operator $P_E: E \to Y$ with $P_E x = x$ for all $x \in E \cap Y$ and $\|P_E\| \le 1 + \varepsilon$. A linear operator $\Phi: Y^* \to X^*$ is called a \emph{Hahn-Banach extension operator}, if $(\Phi y^*)(y) = y^*(y)$ and $\Vert \Phi y^*\Vert =\Vert y^*\Vert $  for all $y \in Y$ and $y^* \in Y^*$.\medskip

N. J. Kalton proved in \cite{Kal} that these three concepts are equivalent and  \AA . Lima showed that $Y$ is a ideal in $X$ if and only if  $Y^{\perp \perp }$ is the range of a norm one projection in $X^{**}$ (see \cite{Lima}).\smallskip

\begin{lemma}\cite{Kal} \cite{Lima}\label{l Kalton} Let $Y$ be a subspace of a real Banach space $X$. The following statements are equivalent:
	\begin{enumerate}
		\item $Y$ is a ideal in $X$.
		\item There exists a Hahn-Banach extension operator $\Phi: Y^* \to
		X^*$.
		\item $Y$ is locally $1$-complemented in $X$.
		\item $Y^{\perp \perp }$ is the range of a norm one projection in $X^{**}$.
	\end{enumerate}
	
\end{lemma}

The next result represents an improvement of Proposition \ref{1-complementado}, since every norm-one complemented subspace of a Banach space is an ideal.

\medskip

\begin{proposition}\label{locally-complemented-ideal}
	Let $X$ be a real Banach space with property $(\overline{\hbox{\rm{P1}}})$. Then every ideal $Y$  in $X$	 has property $(\overline{\hbox{\rm{P1}}})$.
\end{proposition}

\begin{proof}
	
Let $Y$ be an ideal in $X$. By Lemma \ref{l Kalton}, $Y$ is a locally $1$-complemented in $X$ and there exists an extension operator $\Phi : Y^* \to X^*$ with $\|\Phi\| \le 1 $. We claim that there exists $P:X^*\to X^*$ a projection onto $Z:=\Phi (Y^*)$ with $\|P\|=1$ and 	$P^*(X^{**})=Z^*$. Indeed, denote by $R_Y:X^*\to Y^*$ the natural restriction operator, $R_Y(x^*)=x^*|Y$ for $x^*\in X^*$. Then $P=\Phi R_Y:X^*\to X^*$ is a contractive projection on $X^*$ with range $Z$ and $\ker P=Y^\perp$. We also have that $P^*:X^{**}\to Y^{\perp \perp}=Z^*$ is a contractive projection.\smallskip

Let $\{S(B_Y,f_i,\varepsilon_i)\}_{i=1}^k$ be slices of $B_Y$, let $\lambda_i > 0$ with $\sum_{i=1}^k \lambda_i = 1$ and consider the convex combination of these slices $C_Y = \sum_{i=1}^k \lambda_i S(B_Y,f_i,\varepsilon_i)$.\smallskip

Let us consider $\Phi ( f_i)$ be the Hanh-Banach extension of $f_i\in B_{Y^*}$ to $B_{X^*}$. Since $\Phi (f_i)(y)=f_i(y)$ for every $y\in Y$, it is clear that $C_Y = \sum_{i=1}^k \lambda_i S(B_Y,\Phi (f_i),\varepsilon_i)$. We consider the slices	
	\begin{equation*}
		C_{Z^*} = \sum_{i=1}^k \lambda_i S(B_{Z^*},\Phi
		(f_i),\varepsilon_i)\ \ \ \mbox{and} \ \ \  C_{X^{**}} = \sum_{i=1}^k \lambda_i S(B_{X^{**}},\Phi(f_i),\varepsilon_i).
	\end{equation*}

It is clear that $S(B_{Z^*},\Phi(f_i),\varepsilon_i)\subset S(B_{X^{**}},\Phi(f_i),\varepsilon_i)$.\smallskip
	
	Given $z\in S(B_{X^{**}},\Phi(f_i),\varepsilon_i)$, we have that  $\Phi(f_i)(P^*(z))=P(\Phi(f_i))(z)=\Phi
	(f_i)(z)$, and hence $P^*(z)\in S(B_{Z^*},\Phi(f_i),\varepsilon_i)$. This implies that $P^*(S(B_{X^{**}},\Phi(f_i),\varepsilon_i))=S(B_{Z^*},\Phi(f_i),\varepsilon_i)$ and therefore $P^*(C_{X^{**}})=C_{Z^*}$.\smallskip
	
	Let  be $y_0\in C_Y$. Since $X$ has property $(\overline{\hbox{\rm{P1}}})$, by Lemma \ref{P1-bidual}, for $C_{X^{**}} $ and $y_0\in C_{Z^*} \cap Y \subseteq C_{X^{**}}	\cap X$  there exists $W$ relatively $w^*$-open neighborhood of
	$y_0$ in $X^{**}$, such that $y_0\in W\subseteq
	\overline{C}_{X^{**}}^{w^*}$.  Then $W$ contains a set  of the form
	$$V:=\{z\in B_{X^{**}}:\vert g_i(z-y_0)\vert \leq \delta \, \, \, \forall i=1,...,n\},$$ for suitable $\delta >0$, $n\in \N $, and $g_1,...,g_n\in B_{X^*}$. Since $P$ is a norm one projection on $X^*$ with range $Z$, we put $y_i^*\in B_{Y^*}$ such that $\Phi (y_i^*)=P(g_i)$ for every $ i=1,...,n$. We define a relatively weakly open neighborhood of
	$y_0$ in $Y$ by $$U:=\{y\in B_{Y}: \vert y^*_i(y-y_0)\vert <\delta \, \, \, \forall i=1,...,n \}. $$ For $y\in U$ we have that $$ \delta > \vert y^*_i(y-y_0)\vert = \vert \Phi (y_i^*)(y-y_0)\vert = \vert P(g_i)(y-y_0)\vert = \vert g_i(P^*(y-y_0))\vert ,$$ and since $y, y_0\in Y^{**}$,  it follow that $y\in V$. Put	 $$U_1:=\{y\in B_{Y}: \vert \Phi (y_i^*)(y-y_0)\vert <\delta \, \, \, \forall i=1,...,n \},$$ and
	$$U_1^{**}:=\{z\in B_{Z^*}: \vert \Phi (y_i^*)(z-y_0)\vert <\delta \, \, \, \forall i=1,...,n \}.$$
	We consider the topology $w_Z^*:=\sigma (Z^*,Z)$ in $Z^*$. Since $B_Y$ is a norming subset of $B_{Z^*}$ for $Z$, we have that  $B_Y$ is $w^*_Z$-dense in $B_{Z}$, and hence the set
	$U_1$  is $w^*_Z$-dense in $U_1^{**}$. It follows that  $U_1^{**}=	\overline{U_1}^{w^*_Z}\subseteq V \subseteq
	\overline{C}_{X^{**}}^{w^*}$.  This implies that
	\begin{equation*}
		 U_1^{**}=P^*(U_1^{**}) \subseteq P^*(\overline{C}_{X^{**}}^{w^*})=\overline{P^*(C_{X^{**}})}^{w^*}=\overline{C}_{Z^*}^{w^*_Z}.
	\end{equation*}
We have obtained that for $y_0\in C_{Y}$ there exists a relatively weakly open subset $U$ of $B_Y$ wiht $y_0\in U$ and $U=U_1\subseteq \overline{C}_{Z^*}^{w^*_Z}$. Given $y\in U$, since $\overline{C}_{Y}^{w^*_Z}=\overline{C}_{Z^*}^{w^*_Z}$, we have that $y$ is the $w^*_Z$-limit in $Z^*$ of the net $(\sum_{i=1}^n \lambda_i y_{(i,\alpha )})$ with $y_{(i,\alpha )}\in S(B_Y,\Phi (f_i),\varepsilon_i)$.  We recall that $S(B_Y,f_i,\varepsilon_i)=S(B_Y,\Phi (f_i),\varepsilon_i)$ and that  $\Phi (y^*)(y)=y^*(y)$ for all $y\in Y$ and $y^*\in Y^*$. It follows that $y$ is the $w$-limit of $(\sum_{i=1}^n \lambda_i y_{(i,\alpha )})$, and since $C_Y$ is convex subset of $B_Y$, we have that $y\in \overline{C}$. We conclude that $Y$ has property $(\overline{\hbox{\rm{P1}}})$.\end{proof}

\smallskip

As an application of the previous results, we will obtain a characterization of property $(\overline{\hbox{\rm{P1}}})$ in real $L_1$-preduals Banach spaces.

\medskip

\begin{theorem}\label{L_1-predual}
	Let $X$ be a real $L_1$-predual Banach space. Then $X$  contains no isomorphic copy of $\ell_1$ if and only if $X$ has property $(\overline{\hbox{\rm{P1}}})$.
\end{theorem}
\begin{proof} We  assume first that $X$  contains no isomorphic copy of $\ell_1$.  Then, by \cite{Haydon},  $B_{X^*}$ is the closed convex hull of the extreme point of $B_{X^*}$, and hence $X^*$ is purely atomic. This implies that $X^*=\ell _1(\Gamma )$ for some set $\Gamma$. By Theorem \ref{thm-sum}, $X$ has property $(\overline{\hbox{\rm{P1}}})$.
	
	Now suppose that $X$ is not separable, has property $(\overline{\hbox{\rm{P1}}})$  and that $X$ contains an isomorphic copy of $\ell_1$. Then there exists a separable subspace $Y$ of $X$, such that $Y$ is isomorphic to $\ell_1$. Combining \cite[theorem]{Sim-Yost} with Lemma \ref{l Kalton}, there exists a separable ideal $Z$ in $X$ such that $Y\subset Z\subset X$. In a recent work,  P. Bandyopadhyay, S. Duta and A. Sensarma prove that if $X$ is a non-separable $L_1$-predual Banach space, then every separable ideal in $X$ is an $L_1$-predual Banach space (see \cite[Theorem 2.8]{BaDuSe}). By Proposition \ref{locally-complemented-ideal}, we have that $Z$ is a separable $L_1$-predual Banach space with property $(\overline{\hbox{\rm{P1}}})$. Since $Z$ contains an isomorphic copy of $\ell_1$, we have that $Z^*$ is not separable. By a result of A. Lazar and J. Lindenstrauss \cite[Theorem 2.3]{LaLi}, $Z$ contains a norm-one complemented subspace isometric to $C(\Delta )$, the Banach space of continuous function on the Cantor discontinuum $\Delta$. Since property $(\overline{\hbox{\rm{P1}}})$ is inherited by norm-one complemented subspaces (Proposition \ref{1-complementado}), we obtain that $C(\Delta )$ has property $(\overline{\hbox{\rm{P1}}})$. This is a contradiction, since $\Delta$ is a non-scattered  compact topological space (see \cite[Theorem 3.1]{HalKuuPol}).
	
In the case that $X$  is  separable, has property $(\overline{\hbox{\rm{P1}}})$  and that $X$ contains an isomorphic copy of $\ell_1$. We argue similarly to the final part of the previous reasoning.
	
In any case, we conclude that $X$  contains no isomorphic copy of $\ell_1$.\end{proof}

\medskip

\section{Finite dimensional C$^*$-algebras have property $\hbox{\rm{(co)}}$}\label{(co)}

The following result is a generalization of Lemma 2.2 in \cite{AbrLim} to the setting of C$^*$-algebras.

\begin{lemma}\label{l norm-control of convex combinations}
Let $\mathcal{A}$ be a C$^*$-algebra. Let $x$, $y$ be two elements in $B_{\mathcal{A}}$  with $d = \|x+y\|$. Then for every $\lambda\in [0,\frac 12]$ we have $$\|\lambda x+(1-\lambda)y\|\leq\sqrt{1-(4-d^2)(\lambda-\lambda^2)} \leq 1-\frac{(4-d^2)\lambda}{4}$$
\end{lemma}

\begin{proof}
It is well-known that $4\geq d^2=\|x+y\|^2=\|(x+y)^*(x+y)\|=\sup\{ \phi((x+y)^*(x+y)): \phi \in S(\mathcal{A})\}$ where the supremum is taken on the states of $\mathcal{A}$, $S(\mathcal{A})$, and hence $\phi(x^*y+y^*x)\leq d^2-\phi(x^*x)-\phi(y^*y)$ for every $\phi \in S(\mathcal{A})$.\smallskip

Now, it is straightforward to verify that   $$\|\lambda x+(1-\lambda) y\|^2=\|(\lambda x+(1-\lambda) y)^* (\lambda x+(1-\lambda) y)\|$$ $$ = \sup \{\lambda ^2 \phi(x^*x)+ (1-\lambda)^2 \phi(y^*y) +\lambda (1-\lambda)\phi(x^*y+y^*x) : \phi\in S(\mathcal{A})\}$$ $$\leq \sup\{\lambda ^2 \phi(x^*x)+ (1-\lambda)^2 \phi(y^*y) + \lambda (1-\lambda) (d^2-\phi(x^*x)-\phi(y^*y)) : \phi\in S(\mathcal{A})\}$$  $$= \sup\{(\lambda ^2-\lambda (1-\lambda)) \phi(x^*x)+ ((1-\lambda)^2-\lambda (1-\lambda)) \phi(y^*y) + \lambda (1-\lambda) d^2 : \phi\in S(\mathcal{A})\}$$  $$ \leq (\lambda ^2-\lambda (1-\lambda))+ ((1-\lambda)^2-\lambda (1-\lambda))  + \lambda (1-\lambda) d^2= 1-(4-d^2)(\lambda-\lambda ^2).$$ The last inequality follows from the facts  $\sqrt{1+t}\leq 1+\frac{t}{2}$ for $t\geq-1$ and $\lambda-\lambda ^2\geq \frac{\lambda}{2}. $
\end{proof}

\bigskip

Given $e,f$ two partial isometries in a C$^*$-algebra $\mathcal{A}$ we say that $e\leq f $ whenever $ee^*f=e$ and $fe^*e=e$ (equivalently $f=e+(1-ee^*)f(1-e^*e)$). Given a positive element $a$ in the closed unit ball of $\mathcal{A}$ the sequence $(a^n)$ converges to a projection in $\mathcal{A}^{**}$ (in the weak$^*$-topology) called the \emph{support projection} of $a$. This projection can also be define as the biggest projection $p$ in $\mathcal{A}^{**}$ satisfying $p\leq a$. Given $x\in B_{\mathcal{A}}$ with polar decomposition $x=v|x|$ we define, $s(x)$, the \emph{support partial isometry} of $x$ (in $\mathcal{A}^{**}$) as $vp$ where $p$ is the support projection of $|x|$. It is clear that $x=s(x)+(1-s(x)s(x)^*)x(1-s(x)^*s(x))$ and $s(x)$ is the biggest partial isometry in $\mathcal{A}^{**}$ satisfying this property.\smallskip

Given a C$^*$-algebra, $\mathcal{A}$, every partial isometry $e$ in $\mathcal{A}$ induces a decomposition of $\mathcal{A}$ as the direct sum $\mathcal{A}_2(e)\oplus \mathcal{A}_1(e)\oplus \mathcal{A}_0(e)$ where $\mathcal{A}_2(e)=ee^*\mathcal{A}e^*e$, $\mathcal{A}_1(e)=ee^*\mathcal{A}(1-e^*e)\oplus (1-ee^*)\mathcal{A}e^*e$ and $\mathcal{A}_0(e)=(1-ee^*)\mathcal{A}(1-e^*e)$. The corresponding (natural) projections onto these subspaces, $P_i(e)$ $(i\in \{0,1,2\})$, are called \emph{Peirce projections}, are known to be contractive and every element in $\mathcal{A}_2(e)$ is orthogonal to any element in $\mathcal{A}_0(e)$.

\begin{lemma}\label{l support of convex combinations}
Let $\mathcal{A}$ be a C$^*$-algebra and let $x$, $y$ be two elements in $B_{\mathcal{A}}$. Then there exists a partial isometry $e $ in $\mathcal{A}^{**}$ satisfying $$ \lambda x+ (1-\lambda)y= e+ P_0(e)(\lambda x+ (1-\lambda)y) \hbox{ for all } \lambda \in ]0,1[, $$ and being maximal for this property. \smallskip

In particular, when $\mathcal{A}$ is a finite dimensional C$^*$-algebra we have that $e$ is a finite rank partial isometry in $\mathcal{A}$ and $\|P_0(e)(\lambda x+ (1-\lambda)y)\|<1$ for all $ \lambda \in ]0,1[$.
\end{lemma}

\begin{proof}
Fix $\lambda_0\in ]0,1[$ and set $e=s(\lambda_0 x+(1-\lambda_0)y)\in \mathcal{A}^{**}$ whenever $\|\lambda_0 x+(1-\lambda_0)y\|=1$ and $e=0$ in other case. In case $e=0$ the desired equality is trivially satisfied for every $\lambda \in ]0,1[$ thus we can assume that  $\|\lambda_0 x+(1-\lambda_0)y\|=1$ and  $e=s(\lambda_0 x+(1-\lambda_0)y)$.

Since $\mathcal{F}_e=(e+\mathcal{A}_0^{**}(e))\cap \mathcal{B}_{\mathcal{A}}$ is a norm-closed face in the closed unit ball of $\mathcal{A}$ \cite[Theorem 4.10]{AkePed} we have that  $\lambda x+ (1-\lambda)y \in \mathcal{F}_e$ for every $\lambda \in [0,1]$. Therefore $s(\lambda x+ (1-\lambda)y)\geq e=s(\lambda_0 x+(1-\lambda_0)y)$ for every $\lambda \in ]0,1[$. The arbitrariness of $\lambda_0$ gives $s(\lambda x+ (1-\lambda)y)=e$ for every $\lambda \in ]0,1[$ showing the maximality of $e$ among the partial isometries satisfying this property.\smallskip

The final comments should be clear from properties of the support partial isometry of a norm-one element in a finite dimensional C$^*$-algebra (see for example \cite[Theorem 3.1]{Harris81} or \cite[page 19]{Bha}).
\end{proof}

%It can be easily derived from this result above that  $\|\lambda_0 x+(1-\lambda_0)y\|=1$ for %some $\lambda_0\in ]0,1[$ if and only if  $\|\lambda x+(1-\lambda)y\|=1$ for every $\lambda\in %[0,1].$\smallskip

The following result is probably part of the folklore. We include here a proof for the sake of completeness.

\begin{lemma}\label{l geometric properties of partial isometries}
Let $\mathcal{A}$ be a C$^*$-algebra and let $e,f$ be two partial isometries in $\mathcal{A}$ such that $\|e-f\|\leq\varepsilon$, where $\varepsilon$ is positive. Then the following inequalities hold:
\begin{itemize}
\item[(\rm{a)}] $ \|P_2(e)-P_2(f)\|\leq 4 \varepsilon$, $ \|P_1(e)-P_1(f)\|\leq 8 \varepsilon$, $ \|P_0(e)-P_0(f)\|\leq 4 \varepsilon$.
\item [(\rm{b)}] $\|P_k(u)P_j(v)\|\leq 4\varepsilon$ where $u, v \in \{e,f\}$ distinct and $k,j\in \{0,1,2\}$ are distinct.
\end{itemize}
 In particular, given a norm one element $x\in \mathcal{A}$, satisfying $x=e+P_0(e)x$ we have that $ \|P_1(f)x\|,\|P_2(f)x-f\|,\|x- (f+P_0(f)x)\| \leq 5\varepsilon.$

\end{lemma}

\begin{proof}

(\rm{a)} Given $x\in \mathcal{A}$ with $\|x\|=1$, we have that $\|ee^*xe^*e-ff^*xf^*f\|\leq \|(e-f)e^*xe^*e\|+ \|fe^*xe^*e-ff^*xf^*f\| \leq$ $ \varepsilon+ \|f(e^*-f^*)xe^*e\|+ \|ff^*xe^*e-ff^*xf^*f\|\leq 2\varepsilon+ \|ff^*x(e^*-f^*)e\|+ \|ff^*xf^*e-ff^*xf^*f\|\leq  3\varepsilon+ \|ff^*xf^*(e-f)\|\leq 4 \varepsilon$. Since $P_0(e)x=(1-ee^*)x(1-e^*e)$ and $P_1(e)x= (1-ee^*)xe^*e+ee^*x(1-e^*e)$ the rest of inequalities follows in the same manner.

(\rm{b)} $\|P_2(e)P_1(f)\|=\|(P_2(f)-P_2(e))P_1(f)\|\leq 4 \varepsilon$ by \rm{(a)}. The rest of cases can be obtained analogously.

Finally, given $x\in \mathcal{A}$ a norm one element with $x=e+P_0(e)x$, we have that $P_1(f)x=P_1(f)(e-f)+P_1(f)P_2(f)x$, $P_2(f)x-f=P_2(f)(e-f)+P_2(f)P_0(e)x$ and $x-(f+P_0(f)x)=(e-f)+(P_0(e)-P_0(f))x$ thus (\rm{a)} and (\rm{b)} applies.

\end{proof}

The following two remarks contains information concerning perturbation theory in C$^*$-algebras.

\begin{remark}\label{r Ogasawara}

Let $\mathcal{A}$ be a C$^*$-algebra. It is well-known that the absolute value is a continuous function on $\mathcal{A}$, where the absolute value of an element $x$ is the square root of $x^*x$ and is denoted by  $|x|$. Concretely, given $x,y\in \mathcal{A}$ we have that $\| |x|^2-|y|^2 \|=\|x^*x-y^*y\|=\|x^*(x-y)+(x^*-y^*)y\|\leq \|x-y\|(\|x\|+\|y\|)$ and by
\cite[Theorem]{Oga55} $$\| |x|-|y| \|\leq \sqrt{\|x-y\|(\|x\|+\|y\|)}.$$ Therefore we have that small perturbations of an element in a C$^*$-algebra give rise to small perturbations of its absolute value.
\end{remark}\medskip

\begin{remark}\label{r Davis}
Let $\mathcal{A}$ be a finite dimensional C$^*$-algebra. Given $x\in \mathcal{A}$ there exist a unitary $u\in \mathcal{A}$ such that $x=u|x|$. The eigenvalues of $|x|$ are called the \emph{singular values} of $x$ and we may consider them as a vector  $Sing(x)=\{\sigma_1(x),\ldots,\sigma_n(x)\}\in \mathbb{R}^n$, where $n$ is the rank of $\mathcal{A}$, counting this eigenvalues with multiplicity and in decreasing order. \smallskip

Given $x,y\in \mathcal{A}$ there exists a connection between the distance of the corresponding singular values of $x$ and $y$ and the distance between $x$ and $y$. Concretely, $\max\{|\sigma_i(x)-\sigma_i(y)|: i\in \{1,\ldots,n\}\}\leq \|x-y\|$ (see \cite[Theorem 9.8]{Bha}). There is no such a relation, in general, between the distance of the corresponding eigenvectors (see for example \cite[page 46]{BhaDavMcI}). However, for some particular small perturbations of positive elements we get small changes in the corresponding spectral resolutions.\smallskip

More concretely, take a positive element, $h$,  in a finite dimensional C$^*$-algebra $\mathcal{A}$. Let $\beta, \gamma >0$ and assume  $p$ is the spectral resolution of $h$ associated to the set $[\nu, \mu]$ where $\mu-\nu\leq 2\beta$ and the sets $]\nu-\gamma,\nu[$, $]\mu,\mu+\gamma[$ contains no eigenvalues of $h$. Given a positive  $b\in \mathcal{A}$ with $\|b-h\|\leq \delta<\frac{\gamma}{2}$ and $q$ the projection of $b$ associated to $p$ (i.e. $q$ is the spectral resolution of $[\nu-\delta, \mu+\delta])$, C. Davis obtained in \cite[Theorem 2.1]{Dav} that $$ \|q(1-p)\|=\|(1-p)q\|\leq \frac{\beta +\delta}{\beta +\gamma-\delta}.$$ Whenever $\delta<\frac{\gamma}{4}$, applying \cite[Theorem 2.1]{Dav} to $q$ and $p$ with the new parameters  $\beta'=\beta+\delta$ and $\gamma'=\frac{\gamma}{2}$ we have that $$\|p(1-q)\|=\|(1-q)p\|\leq \frac{\beta +2\delta}{\beta +\frac{\gamma}{2}}.$$ In the particular case of $\beta=0$ ($p$ is the spectral resolution of a single eigenvalue) and $\delta<\frac{\gamma}{4}$ we have that $$\|p-q\|=\|p(1-q)-(1-p)q\|\leq \|p(1-q)\|+\|(1-p)q\|\leq  $$ $$\frac{\delta}{\gamma-\delta}+ \frac{4\delta}{\gamma}\leq  \frac{16}{3} \frac{\delta}{\gamma}$$ (compare with \cite[Theorem 2]{Vas} where a different bound is given).

\end{remark}\medskip

For our purposes, given an element $x$ in a finite dimensional C$^*$-algebra, $\mathcal{A}$, of rank $n$, it will be more convenient to express $x$ as the sum  $\sum_{i=1}^{n_0} \sigma_i(x) e_i$ where $\{\sigma_i(x): i\in \{1,\ldots,n_0\}\}$ are the singular values of $x$ (eigenvalues of $|x|$) taken in decreasing order not counting multiplicity. Moreover, $x=u|x|$ with $u$ unitary in $\mathcal{A}$,  $e_i=up_i$ is a finite rank partial isometry in $\mathcal{A}$ with $p_i$ the spectral resolution of $|x|$ associated to the singular value $\sigma_i(x)$ and $\sum_{i=1}^{n_0} rank(e_i)=n$ (see for example \cite[Theorem 3.1]{Harris81}).\smallskip

Having in mind Remark \ref{r Davis} above, we have that, once fixed a positive $\delta$ smaller than one half the distance between the elements of the union of the singular values of $x$ and 0, for every element $y$ satisfying $\|y-x\|\leq \delta$, associated to each $p_i$ we have a projection $q_i$, the spectral resolution of $|y|$ with respect to the set $[\sigma_i(x)-\delta,\sigma_i(x)+\delta]$. Notice that $|y|$ does not coincide in general with a linear combination of this projections. Therefore, for each $i\in \{1,\ldots,n_0\} $, we have associated a partial isometry  $f_i=vq_i$, where  $v$ is the unitary such that $y=v|y|$. \medskip

The following result exhibits the continuity at some fixed point of the perturbed spectral resolutions.

\begin{theorem}\label{t continuity of spectral resolutions}
Let $\mathcal{A}$ be a finite dimensional C$^*$-algebra. Let $ x$ be an element in $\mathcal{A}$. Then for every positive $\varepsilon$ there exists $\delta>0$ such that for every $y$ in the closed ball centered at $x$ of radius $\delta$, $\|e_i-f_i\|\leq\varepsilon $ for each $i\in \{1,\ldots,n_0\}$ such that $\sigma_i(x)>0$.

\end{theorem}
\begin{proof}

We define $\gamma$ as the distance between the elements of the union of the singular values of $x$ and 0. Let $\delta$ be a positive number satisfying $$\max\{\frac{16}{3}\frac{\sqrt{(2\|x\|+\delta)\delta}}{\gamma}+\frac{1}{\sigma_i(x)}(\sqrt{(2\|x\|+\delta)\delta}+ \delta) : \sigma_i(x(>0\}\leq \varepsilon. $$

Take $y\in \mathcal{A}$ with $\|x-y\|\leq \delta$. Let $n_0$, $e_i$, $f_i$, $p_i$ and $q_i$ be defined as in the comments preceding this theorem. By Remark \ref{r Ogasawara} we have that $\| |x|-|y| \| \leq \sqrt{(2\|x\|+\delta)\delta}$. Now,  Remark \ref{r Davis} gives $\displaystyle \|p_i-q_i\| \leq \frac{16}{3}\frac{\sqrt{(2\|x\|+\delta)\delta}}{\gamma}$ $\forall i\in \{1,\ldots,n_0\}$. By polar decomposition we have that  $x=u|x|$ and $y=v|y|$ where $u$ and $v$ are unitaries in $\mathcal{A}$ with $e_i=up_i$ and $f_i=vq_i$.\smallskip

Clearly, $|x|p_i=\sigma_i(x) p_i$ and it is straightforward to check that  $\||x|-v^*x\|=\| |x|-|y|+|y|-v^*x \| \leq \| |x|-|y|\|+ \| |y|-v^*x \| \leq \| |x|-|y|\| + \| y-x\|\leq \sqrt{(2\|x\|+\delta)\delta}+ \delta$.\smallskip

Moreover, for each $\sigma_i(x)>0$ we have $$ \|f_i-e_i \|=\|q_i-v^*up_i \|=\|q_i-p_i+p_i-v^*up_i\|\leq \|q_i-p_i\| + \|p_i-v^*up_i\|\leq $$ $$ \|q_i-p_i\| +\|(|x|-v^*u|x|)\frac{p_i}{\sigma_i(x)}\|\leq \|q_i-p_i\| +\frac{1}{\sigma_i(x)}\||x|-v^*x\|  \leq $$ $$ \frac{16}{3}\frac{\sqrt{(2\|x\|+\delta)\delta}}{\gamma}+\frac{1}{\sigma_i(x)}(\sqrt{(2\|x\|+\delta)\delta}+ \delta)\leq \varepsilon.$$

\end{proof}\medskip

It would be convenient to highlight and isolate a particular case of Theorem \ref{t continuity of spectral resolutions}.

\begin{remark}\label{r perturbation of supports}
 Let $\mathcal{A}$ be a finite dimensional C$^*$-algebra. Let $x, y$ be elements in the closed unit ball of $\mathcal{A}$. Assuming that $\|x\|=1$, we denote by  $e=s(x)$ the support partial isometry of $x$ (i.e. $\gamma_1(x)=1$ and $e=e_1$), and $\gamma =1 -\|x-s(x)\|$. Let $\delta$ be a positive number with $\delta<\sqrt{2\delta}<\frac{\gamma}{4}$ and suppose that $\|x-y\|\leq \delta$. Denoting by $f$ the spectral resolution of $y$ corresponding to the set $[1-\delta,1]$ we have that  $$\|e-f\|\leq \delta+\sqrt{2\delta}+\frac{16}{3}\frac{\sqrt{2\delta}}{\gamma}.$$
\end{remark}

We are now in a position to prove the main result of this section. The proof is highly influenced by the proof of Theorem 2.3 in \cite{AbrLim}.\medskip

\begin{theorem}\label{t Mn is (co)}
Every finite dimensional C$^*$-algebra  has property $\hbox{\rm{(co)}}$.
\end{theorem}

\begin{proof}

Let $\mathcal{A}$ be a finite dimensional C$^*$-algebra, $n \in \mathbb{N}$ . Let $x_1, \ldots, x_n$ be elements in the closed unit ball of $\mathcal{A}$ and let $\lambda_1,\ldots,\lambda_n$ be positive numbers with $\displaystyle \sum_{i=1}^{n} \lambda_i=1$. We claim that for every positive $\varepsilon$ there exist a positive $\delta$ such that given $y\in B_\mathcal{A}$ with $\displaystyle \|y-\sum_{i=1}^{n} \lambda_i x_i\|\leq \delta$, there exist $\tilde{x}_1,\ldots, \tilde{x}_n$ in $ B_\mathcal{A}$ satisfying $\displaystyle y=\sum_{i=1}^{n} \lambda_i \tilde{x}_i$ and $\|x_i-\tilde{x}_i \|\leq \varepsilon$, for each $i\in \{1,\ldots,n\}$.\medskip

Assume first that $\|\displaystyle \sum_{i=1}^{n} \lambda_i x_i\|=1$.\medskip

In this case, by Lemma \ref{l support of convex combinations}, denoting by $e$ the support partial isometry of $ \sum_{i=1}^{n} \lambda_i x_i$, we have that $e\neq 0$ and $e$ is also the support partial isometry of any other (strict) convex combination of the elements $\{x_1,\ldots,x_n\}$.\smallskip

We set, for each $j\in \{1,\ldots,n\}$, $\displaystyle a_j=\sum_{i=1, i\neq j}^{n} \frac{x_i}{n-1} \in B_{\mathcal{A}}$. We define $d=\max \{\|P_0(e)(a_j+x_j)\|: j\in \{1,\ldots,n\}\}$. It should be clear from Lemma \ref{l support of convex combinations} that $d<2$. It is also direct to verify that $\displaystyle \sum_{j=1}^{n} (a_j-x_j)=0$ and $\lambda a_j+(1-\lambda)x_j=e+P_0(e)(\lambda a_j+(1-\lambda)x_j) $ for every $\lambda \in [0,1]$.\smallskip

Fix now $c>0$ satisfying $0 < c\leq \frac{\varepsilon}{4} \min \{\lambda_i:i\in \{1,\ldots,n \}\}$ and define $\mu_j=\displaystyle \frac{c}{\lambda_j}$. It is clearly satisfied that \begin{equation}\label{eq mu} \max \{ \mu_j : j \in \{1,\ldots,n \}\} =\frac{c}{\min \{\lambda_j:j\in \{1,\ldots,n \}\}} \leq \frac{\varepsilon}{4}. \end{equation}

We set $\gamma=1-\|P_0(e)(\sum_{i=1}^{n} \lambda_i x_i)\|$, which is positive by Lemma \ref{l support of convex combinations}.\smallskip

We can associate to every positive $\delta$, the following positive number  $\varepsilon_1=\varepsilon_1(\delta)=\delta+\sqrt{2\delta}+\frac{16}{3}\frac{\sqrt{2\delta}}{\gamma}$, which satisfies $\lim_{\delta \to 0} \varepsilon_1(\delta)=0$.\smallskip

Take $\delta>0$ satisfying $\delta<\sqrt{2\delta}<\frac{\gamma}{4}$,  \begin{equation}\label{eq 3} 4 \varepsilon_1+\delta< \min \{\lambda_j:j\in \{1,\ldots,n \}\} (\frac{4-d^2}{4})
 \end{equation} and \begin{equation}\label{eq 4} 10\varepsilon_1 + 2\delta < \frac{\varepsilon}{2}.
 \end{equation}

Applying Remark \ref{r perturbation of supports} to any $y$ in the closed unit ball of $A$ with $\|\sum_{i=1}^{n} \lambda_ix_i-y\|<\delta$ and denoting by $f$ the spectral resolution of $y$ associated to the set $[1-\delta,1]$, we have that \begin{equation}\label{eq 5}  \|f-e\|\leq \varepsilon_1.
 \end{equation}

We define next the elements $\tilde{x}_j$ and check the desired statements.\smallskip

For each $j\in \{1,\ldots, n\}$, we define $$\tilde{x}_j = P_2(f)y+P_0(f)[x_j+\mu_j(a_j-x_j)+y-\sum_{i=1}^{n} \lambda_i  x_i].$$ It follows straightforwardly that, $$ \sum_{j=1}^{n}  \lambda_j \tilde{x}_j = \sum_{j=1}^{n} \lambda_j P_2(f)y+P_0(f)[\sum_{j=1}^{n} \lambda_j x_j+\sum_{j=1}^{n} \lambda_j\mu_j(a_j-x_j) + \sum_{j=1}^{n} \lambda_j y- \sum_{j=1}^{n} \lambda_j \sum_{i=1}^{n}\lambda_i  x_i]=$$ $$ P_2(f)y+P_0(f)[\sum_{j=1}^{n} \lambda_j x_j + c \sum_{j=1}^{n} (a_j-x_j) +  y-  \sum_{i=1}^{n}\lambda_i  x_i]=P_2(f) y+P_0(f)y=y.$$
It is also satisfied that $\|x_j-\tilde{x}_j\|\leq \varepsilon$ for every $j\in \{1,\ldots,n\}$. Indeed, remembering that $\|\sum_{i=1}^{n} \lambda_ix_i-y\|<\delta$, we have $$\|x_j-\tilde{x}_j\|=\|x_j-P_2(f)y-P_0(f)[x_j+\mu_j(a_j-x_j)+y-\sum_{i=1}^{n} \lambda_i  x_i]\|= $$ $$\|P_2(f)x_j+P_1(f)x_j+P_0(f)x_j-P_2(f)y+f-f- P_0(f)x_j - \mu_jP_0(f)a_j+\mu_j P_0(f)x_j + $$ $$ P_0(f)(y-\sum_{i=1}^{n} \lambda_i  x_i)\|\leq \|P_2(f)x_j-f \| + \|f-P_2(f)y \| + \|P_1(f)x_j \| + \mu_j\|P_0(f)(x_j-a_j) \|+$$ $$ \|P_0(f)((y-\sum_{i=1}^{n} \lambda_i  x_i) \|\leq (\hbox{by (\ref{eq 5}), Lemma \ref{l geometric properties of partial isometries} and the definition of $f$ })\leq $$ $$5\varepsilon_1 + \delta + 5\varepsilon_1 + 2\mu_j+\delta \leq (\hbox{by (\ref{eq mu}) and (\ref{eq 4}) }) \leq \varepsilon.$$

Finally we will show that $\|\tilde{x}_j\|\leq 1$ for every $j\in \{1,\ldots,n\}$. Since $\|\tilde{x}_j\|=\max \{\|P_2(f)y\|,\|P_0(f)[x_j+\mu_j(a_j-x_j)+y-\sum_{i=1}^{n} \lambda_i  x_i] \| \},$ we only have to check that the second summand is less than  or equal to 1. Now  $$\|P_0(f)[x_j+\mu_j(a_j-x_j)+y-\sum_{i=1}^{n} \lambda_i  x_i] \|\leq \|P_0(f)[(1-\mu_j)x_j+\mu_j a_j]\|+\|y-\sum_{i=1}^{n} \lambda_i  x_i\|\leq $$ $$\|P_0(e)[(1-\mu_j)x_j+\mu_j a_j]\|+\|(P_0(e)-P_0(f))[(1-\mu_j)x_j+\mu_j a_j]\|+\|y-\sum_{i=1}^{n} \lambda_i  x_i\| \leq $$ $$(\hbox{by Lemma \ref{l norm-control of convex combinations} and Lemma \ref{l geometric properties of partial isometries} \rm{a)}} ) \leq 1-\frac{4-d^2}{4}\mu_j+ 4\varepsilon_1+\delta \leq (\hbox{by (\ref{eq 3})}) \leq 1. $$

The case $\|\displaystyle \sum_{i=1}^{n} \lambda_i x_i\|<1$ is even simpler. Notice that in this case $e=0$ so that $P_2(e)=P_1(e)=0$ and $P_0(e)=\hbox{Id}_{|\mathcal{A}}$ and if  $\delta<\frac{\gamma}{4}$, the spectral resolution of $y$ corresponding  to the set $[1-\delta,1]$, $f$, is also zero. Defining $\tilde{x_j}$ in the same manner, with the less restrictive assumption $\delta\leq \min \{\frac{\varepsilon}{2}, \min \{ \frac{4-d^2}{4} \mu_j: j\in \{1,\ldots,n\} \}\}$ we arrive to the desired conclusion.\medskip

In order to prove that $\mathcal{A}$ has Property $\hbox{\rm{(co)}}$ (see Definition \ref{d (co)}) and once we have fixed $\delta>0$, we only have to check that the functions $\phi_j:B(x,\delta)\cap B_{\mathcal{A}}\to B(x_j,\varepsilon)\cap B_{\mathcal{A}}$ defined by $\phi_j(y)=\tilde{x}_j$ for every $j\in \{1,\ldots,n\}$ are continuous.\smallskip

Given $y,z$ in $B(x,\delta)\cap B_{\mathcal{A}}$, we have that $$ \|\phi_j(y)-\phi_j(z)\|=\| P_2(f_y)y+P_0(f_y)[x_j+\mu_j(a_j-x_j)+y-\sum_{i=1}^{n} \lambda_i  x_i]- $$ $$ P_2(f_z)z-P_0(f_z)[x_j+\mu_j(a_j-x_j)+z-\sum_{i=1}^{n} \lambda_i  x_i] \|=$$ $$ \|y-z+ (P_0(f_y)-P_0(f_z))[x_j+\mu_j(a_j-x_j)-\sum_{i=1}^{n} \lambda_i  x_i]\|\leq \|y-z\|+8\|f_y-f_z \|, $$ where in the last inequality we have used Lemma \ref{l geometric properties of partial isometries} and $\|x_j+\mu_j(a_j-x_j)-\sum_{i=1}^{n} \lambda_i  x_i \|\leq 2$. Theorem \ref{t continuity of spectral resolutions} assures that the functions $\phi_j$ are continuous.

\end{proof}

\section{Scattered C$^*$-algebras and Properties $(P1)$ and $(\overline{\hbox{P1}})$}
\label{sect: scattered}
\medskip

A topological space $K$ is called \emph{scattered} if each closed subset of $K$ has an isolated point. It is well known that a locally compact Hausdorff topological space is scattered if and only if there exists no non-zero atomless regular Borel measure on $K$ (see \cite{PelSem} and \cite{Lut70}).\smallskip

A C$^*$-algebra $\mathcal{A}$ is said to be a \emph{scattered C$^*$-algebra} if every positive functional on $\mathcal{A}$ is the sum of a sequence of pure functionals \cite{Jen77}. Clearly, an abelian $C^*$-algebra is scattered if and only if the space of characters (equipped with the weak$^*$ topology) is scattered. There are several characterizations of scattered C$^*$-algebras by different authors. For example, they are Type I with scattered spectrum \cite[Corollary 3]{Jen78}, its bidual coincides with an $\ell_{\infty}$-sum of factors of type I \cite[Theorem 2.2]{Jen77}, it admits a composition series such that every quotient algebra is elementary \cite[Theorem 2]{Jen78}, its dual has the Radon-Nikod\'{y}m property \cite[Theorem 3]{Chu81} and every C$^*$-subalgebra has real rank-zero \cite[Theorem 2.3]{Kus}. However, we will take advantage of the characterization obtained by T. Huruya \cite[Theorem]{Hur} which assures that a C$^*$-algebra is scattered if and only if the spectrum of every hermitian element is scattered (i.e. countable).
\medskip

Recently,  T.A. Abrahamsen and V. Lima and R. Haller, P. Kuuseok and M. P\~{o}ldvere show that given $K$ a  compact Hausdorff topological space, the space of continuous $\K$-valued  functions on $K$,  $C(K, \K)$, has property $\hbox{\rm{(P1)}}$ if and only if $K$ is scattered (see \cite{AbrLim, HalKuuPol}). The main goal of this section is to present connections between properties (P1) and $(\overline{\hbox{P1}})$ and being scattered in the setting of general C$^*$-algebras.\smallskip

\begin{remark}\label{r (P1) implies K scattered revisited}
In \cite{HalKuuPol} the authors show that for every non-scattered locally compact Hausdorff space $K$, there exist a convex combination of slices in the unit ball of $\mathcal{A}=C_0(K,\mathbb{R})$ with empty interior. This statement is stronger than  $C_0(K,\mathbb{R})$ fails property $\hbox{\rm{(P1)}}$. If we only want to show the latter result, the (same) arguments given in the proof of Theorem 3.1 in \cite{HalKuuPol} can be shortened. Concretely, since $K$ is not scattered  there exists an atomless measure $\mu$ with $\mu(K)=1$. Let $\varepsilon \in ]0,1[$ and let us define $S_1=S(B_{\mathcal{A}},\mu,\varepsilon)=\{x\in  C_0(K,\mathbb{R}): \|x\|\leq 1,\, \mu(x)>1-\varepsilon\}$, $S_2=-S_1$ and $C=\frac{S_1+S_2}{2}$. Clearly $0$ belongs to $C$ and for each weakly open neighborhood of $0$, $\mathcal{U}$,  there exist disjoint compact sets $K^1, K^2$ contained in $K$, a positive $\delta$ such that  $3\delta+2\sqrt{\delta}< 1-\varepsilon $, $|\mu(K^1)-\mu(K^2)|< \delta$, $\mu(K\backslash (K^1\cup K^2))<\delta$, and a continuous function in the unit sphere of $C_0(K,\mathbb{R})$, $y_{_\mathcal{U}}$, satisfying $y_{_\mathcal{U}} \in \mathcal{U}$, $y_{_\mathcal{U}}(t)=1 \;\;\; \forall t\in K^1$, $y_{_\mathcal{U}}(t)=-1 \;\;\; \forall t\in K^2$, $y_{_\mathcal{U}}\notin C$.
\end{remark}

\begin{proposition}\label{p_(P1) implies scattered}
Every C$^*$-algebra with property $\hbox{\rm{(P1)}}$  or $(\overline{\hbox{\rm{P1}}})$ is scattered.
\end{proposition}

\begin{proof}

Let $\mathcal{A}$ be a C$^*$-algebra satisfying property  $(\overline{\hbox{P1}})$. $\mathcal{A}$ will be scattered the moment we show that the spectrum of every self-adjoint element is scattered.\smallskip

Assume that there exists $h\in \mathcal{A}_{sa}$ with non-scattered spectrum. Let us denote by $C_h$ the abelian C$^*$-subalgebra generated by $h$ which is isomorphic  to $C_0(Sp(h),\mathbb{C})$ (\cite[Proposition 1.1.8 ]{Ped}). Since the spectrum of $h$ is non-scattered  there exists an atomless regular Borel measure on $Sp(h)$, $\mu$, with $\mu(Sp(h))=1$. Fix $0<\varepsilon <1$ and define $S_1=S(B_{\mathcal{A}},\mu,\varepsilon)=\{z\in  B_{\mathcal{A}}:  \mathfrak{Re}\mu(z)>1-\varepsilon\}$, $S_2=-S_1$ and $S=\frac{S_1+S_2}{2}$. Clearly $0$ belongs to $S$ and we will show that every relatively weakly open subset of the closed unit ball of $\mathcal{A}$ containing 0 has an element not contained in $S$.\smallskip

Let $\mathcal{W}=\{z\in B_\mathcal{A}: |\varphi_i(z)|<\gamma, i=1,\ldots,n\}$ where $\varphi_i\in \mathcal{A}^*$ ($i=1,\ldots,n$) and $\gamma>0$. Let us denote by $f_i$ the restriction of $\varphi_i$ to $C_h$. Clearly $\mathcal{U}=\{z\in B_{C_h}: |f_i(z)|<\gamma, i=1,\ldots,n\} \subseteq \mathcal{W}$ is a neighborhood of 0 in the weak topology of $C_h$ restricted to its unit ball. By Remark \ref{r (P1) implies K scattered revisited} there exist disjoint compact sets $K^1, K^2$ contained in $Sp(h)$, a positive $\delta$ such that  $3\delta+2\sqrt{\delta}< 1-\varepsilon $,  $|\mu(K^1)-\mu(K^2)|< \delta$, $\mu(K\backslash (K^1\cup K^2))<\delta$, and an element in the unit sphere of $C_h$, $y_{_\mathcal{U}}$, satisfying $y_{_\mathcal{U}} \in \mathcal{U}$, $y_{_\mathcal{U}}(t)=1 \;\;\; \forall t\in K^1$, $y_{_\mathcal{U}}(t)=-1 \;\;\; \forall t\in K^2$.\smallskip

We will consider $\mu$ both as a measure on $Sp(h)$ and as a functional on $\mathcal{A}^{**}$ via the identification of the sets contained in $Sp(h)$ and their associated characteristic functions in $\mathcal{A}^{**}$. Under this considerations, we denote by $p_1$ (respectively, $p_2$) the projection in $\mathcal{A}^{**}$ given by the characteristic function associated to the set $K^1$ (respectively, $K^2$) and we set $p=p_1+p_2$. Clearly $p_1$ and $p_2$ are orthogonal,  $\mu(p_j)=\mu(K^j)$ ($j=1,2$) and $|\mu(p_1)-\mu(p_2)|<\delta$. Having in mind that $\mu(Sp(h))=1$ we have that $$1=\mu(Sp(h))=\mu(K^1)+\mu(K^2)+\mu(Sp(h) \backslash (K^1\cup K^2)) < \mu(K^1)+\mu(K^2)+\delta,$$ and consequently $\mu(p)=\mu(p_1)+\mu(p_2)>1-\delta$.\medskip

The M-orthogonality between $p\mathcal{A}^{**}p$ and $(1-p)\mathcal{A}^{**}(1-p)$ together with $\mu(p)>1-\delta$ assures that $\|\mu_{|_{(1-p)\mathcal{A}^{**}(1-p)}}\|< \delta$. Moreover, $\mu(1-p)<\delta$ and by the Cauchy-Schwarz inequality \cite[Theorem 3.1.3]{Ped} we have that $|\mu((1-p)z)|, |\mu(z(1-p))|<\sqrt{\delta}  $ for every $z\in B_{\mathcal{A}}$.\smallskip

Suppose now that $y_{_\mathcal{U}}$ belongs to $\bar{S}$, so that $y_{_\mathcal{U}}=\lim\frac{x_n+y_n}{2}$ with $x_n\in S_1$, $y_n\in S_2$. Denoting by $e=p_1-p_2$ (a symmetry  in the C$^*$-algebra $p\mathcal{A}^{**}p$),  simple arguments with respect to the order in $p\mathcal{A}^{**}p$ gives $e=\lim px_np$ and $e=\lim py_np$. Therefore, for every natural $n$ with $\|e-px_np\|\leq \delta$ we have that $$ |\mu(x_n)|\leq |\mu(px_np-e)|+|\mu(e)|+|\mu((1-p)x_np)|+|\mu(px_n(1-p))|+$$ $$|\mu((1-p)x_n(1-p))|\leq 3\delta+2\sqrt{\delta}< 1-\varepsilon $$ which gets in contradiction with $x_n\in S_1$.

\end{proof}

Notice that there exists scattered C$^*$-algebras not satisfying property $(\overline{\hbox{\rm{P1}}})$. For example, it is straightforwardly to check that the (scattered) C$^*$-algebra of compact operators on an infinite dimensional Hilbert space do not satisfy property $(\overline{\hbox{\rm{P1}}})$.\smallskip

We recall that a \emph{compact C$^*$-algebra} is the $c_0$-sum of algebras of compact operators on a Hilbert space (see \cite[Theorem 8.2]{Ale}). The following result is a generalization of  \cite[Theorem 2.3]{AbrLim} and \cite[Theorem 3.1]{HalKuuPol}.

\begin{theorem}\label{c C(K,Mn) has P1}
Let $K$ be a locally compact Hausdorff topological space and let $\mathcal{A}$ be a compact C$^*$-algebra. The following assertions are equivalent:

\begin{enumerate}
\item[i)] $C_0(K,\mathcal{A})$ has property $\hbox{\rm{(P1)}}$

\item[ii)] $C_0(K,\mathcal{A})$ has property $(\overline{\hbox{\rm{P1}}})$

\item[iii)] $K$ is scattered and $\mathcal{A}$ is the $c_0$-sum of finite-dimensional C$^*$-algebras.

\end{enumerate}

\end{theorem}

\begin{proof}

Clearly \textit{i)} implies \textit{ii)}.\smallskip

To show that \textit{ii)} implies \textit{iii)}, let us assume that  $C_0(K,\mathcal{A})$ has property $(\overline{\hbox{\rm{P1}}})$. It is well known that $C_0(K)$ is isometrically isomorphic to a norm-closed one-complemented subspace of $C_0(K,\mathcal{A})$. By Proposition  \ref{1-complementado} we have that $C_0(K)$ has property $(\overline{\hbox{\rm{P1}}})$ and Proposition \ref{p_(P1) implies scattered} assures that $K$ is scattered. Since $K$ has isolated points we deduce that $\mathcal{A}$ is also one-complemented subspace of $C_0(K,\mathcal{A})$ and hence has property $(\overline{\hbox{\rm{P1}}})$, which gives that $\mathcal{A}$ is the $c_0$-sum of finite-dimensional C$^*$-algebras.\smallskip

Finally, if $K$ is scattered and $\mathcal{A}$ is the $c_0$-sum of finite-dimensional C$^*$-algebras, $C_0(K,\mathcal{A})$ has property $\hbox{\rm{(P1)}}$ by Corollary \ref{slices-for-C_0} and Theorem \ref{t Mn is (co)}.
\end{proof}

Given a C$^*$-algebra, $\mathcal{A}$, there exists a factorization of its bidual as the sum of its atomic and non-atomic ideals. More concretely, the atomic representation of $\mathcal{A}$ is an $\ell_{\infty}$-sum  of $B(H_{\pi})$ where $H_{\pi}$ are Hilbert spaces associated to (unitarily equivalent)  irreducible representations and the sum is taken over the spectrum of $\mathcal{A}$ (see \cite[4.3.7 and 4.3.8]{Ped}). We say that every factor in the atomic decomposition of $\mathcal{A}$ is finite dimensional whenever $dim(H_{\pi})<\infty$ for each irreducible representation.

\begin{proposition}\label{p_(P1) and Cartan factors}
Let $\mathcal{A}$ be a C$^*$-algebra  satisfying property $\hbox{\rm{(P1)}}$  or $(\overline{\hbox{\rm{P1}}})$. Then every factor appearing in the atomic decomposition of $\mathcal{A}$ is finite dimensional.
\end{proposition}

\begin{proof}

Assume that there exists an irreducible representation of $\mathcal{A}$ over $B(H)$, with $H$  infinite dimensional Hilbert space. Fix a norm one vector $\eta_0 \in H$ and an orthonormal system $\{\eta_i:i\in I\}$ such that $\{\eta_0\}\cup \{\eta_i:i\in I\}$ is a basis of $H$. We define a minimal projection $p=\eta_0\otimes \eta_0$ and rank one partial isometries $u_i=\eta_i \otimes \eta_0$ in $\mathcal{A}^{**}$. Clearly, $pu_i=u_i$, $u_ip=0$, and $u_iu_i^*p=u_iu_i^*=p$ for every $i\in I$. Moreover, the closure of the linear span of $\{p\}\cup\{u_i:i\in I\}$ is isometric to the Hilbert space $H$ (i.e. $p\mathcal{A}^{**}=pB(H)\cong H$) and hence for every $\varphi \in \mathcal{A}^*$ and $\delta >0$ the set $\{i\in I: |\varphi(u_i)|\geq \delta\}$ is finite.\medskip

We denote by $\varphi_0$ the support functional associated to the minimal projection $p$ (the extreme point in the unit ball of $\mathcal{A}^*$ satisfying $\varphi_0(p)=1$). More concretely, $ \varphi_0(z)=\varphi_0(pzp)$ for every $z\in \mathcal{A}^{**}$ and hence $\|x\|=|\varphi_0(x)|=\varphi_0(x^*x)^{\frac 12}$ for all $x\in p\mathcal{A}^{**}p$.\smallskip

Given $0<\varepsilon<1$, we define $S_1=S(B_{\mathcal{A}},\varphi_0,\varepsilon)=\{x\in B_{\mathcal{A}}: \mathfrak{Re}\varphi_0(x)>1-\varepsilon\}$, $S_2=-S_1$ and $S=\frac{S_1+S_2}{2}$.\smallskip

For each $x$ in $S_1$ we have that $\|px\|^2=\|pxp\|^2+\|px(1-p)\|^2$ thus $\|px(1-p)\|\leq (1-\|pxp\|^2)^{\frac 12 } = (1-|\varphi_0(x)|^2)^{\frac 12 } \leq (1-(1-\varepsilon)^2)^{\frac 12 }$ and the same happens with elements $y$ in $S_2$, thus
\begin{equation}\label{eq norms in S} \|pz(1-p)\|\leq  (1-(1-\varepsilon)^2)^{\frac 12 } <1 \hbox{ for all } z \hbox{ in } \bar{S}.
 \end{equation}

Given a relatively weakly open subset of $B_{\mathcal{A}}$ containing 0, $U=\{x\in B_{\mathcal{A}}: |\varphi_j(x)|<\delta, j=1,\ldots,n\}$ where $\varphi_j\in \mathcal{A}^*$ for every $j=1,\ldots,n$ and $\delta$ is a positive number, we will find a norm-one element  $a\in U$ satisfying $\|pa(1-p)\| =1 $ so that $a\notin \bar{S}$ by (\ref{eq norms in S}).  \medskip

As claimed at the beginning of the proof there exists $u\in \{u_i:i\in I\}$ satisfying $|\varphi_j(u)|<\frac{\delta}{2} $ for every $j=1,\ldots,n$. We recall that, by Kadison's transitivity theorem,  rank one partial isometries in $\mathcal{A}^{**}$ are compact (belong locally to $\mathcal{A}$ in the terminology of \cite{AkePed}) thus  there exists a net $(a_{\lambda})$ in $B_\mathcal{A}$ converging in the weak$^*$-topology to $u$ and satisfying $(a_{\lambda})=u+(1-uu^*)(a_{\lambda})(1-u^*u)$ for every $\lambda$ (see \cite{AkePed} or \cite[Theorem 5.1]{EdwRut}).\smallskip

The weak$^*$-convergence of $(a_{\lambda})$ to $u$ assures the existence of $\lambda_0$ such that $|\varphi_j(a_{\lambda_0})|<\delta $ for every $j=1,\ldots,n$. Therefore, denoting by $a=a_{\lambda_0}$ we have that $a\in U$ and $a=u+(1-uu^*)a(1-u^*u)$. It is straightforward to verify that $$\|pa(1-p)\|= \|p(u+(1-uu^*)a(1-u^*u))(1-p) \|=\|u+p(1-uu^*)a(1-u^*u)(1-p)\|=$$ $$\|u+(1-uu^*)pa(1-p)(1-u^*u)\|= 1,$$ which gives a contradiction with (\ref{eq norms in S}).
\end{proof}

Our final result will be a  complete characterization of C$^*$-algebras satisfying property $(\overline{\hbox{\rm{P1}}})$.\smallskip

\begin{theorem}\label{t caracterizacion P1 barra}
 A C$^*$-algebra satisfy property $(\overline{\hbox{\rm{P1}}})$ if and only if it is scattered with finite dimensional irreducible representations.

\end{theorem}

\begin{proof}
The only if part was given in Proposition \ref{p_(P1) implies scattered} and Proposition \ref{p_(P1) and Cartan factors} while the if part is a consequence of Theorem \ref{t Mn is (co)} and Theorem \ref{thm-sum}.
\end{proof}

\bigskip
\textbf{Acknowledgements} The first author has been partially supported by Junta de Andaluc\'ia grant FQM199 and  Spanish government grant MTMT2016-76327-C3-2-P. The second author has been partially supported by
Junta de Andaluc\'{\i}a grant FQM375.\smallskip

\bibliographystyle{amsplain}

\end{document}